\pdfoutput=1

\documentclass[a4paper,UKenglish]{amsart}

\usepackage[utf8]{inputenc}
\usepackage[T1]{fontenc}
\usepackage{lmodern}
\usepackage{mathtools}
\usepackage{amssymb}
\usepackage{tikz-cd}
\usepackage{babel}
\usepackage{microtype}
\usepackage{hyperref}

\DeclareFontFamily{OMX}{lmex}{}
\DeclareFontShape{OMX}{lmex}{m}{n}{<->lmex10}{}

\tikzcdset{arrow style=Latin Modern}

\theoremstyle{plain}
\newtheorem{theo}{Theorem}
\newtheorem{coro}[theo]{Corollary}
\newtheorem{prop}[theo]{Proposition}
\newtheorem{lemm}[theo]{Lemma}

\theoremstyle{remark}
\newtheorem*{rema}{Remark}
\newtheorem*{ques}{Question}

\DeclareMathOperator{\Hom}{Hom}
\DeclareMathOperator{\Ext}{Ext}
\DeclareMathOperator{\Mod}{Mod}
\DeclareMathOperator{\Ind}{Ind}
\DeclareMathOperator{\Ord}{Ord}
\DeclareMathOperator{\Lr}{L}
\DeclareMathOperator{\Rr}{R}
\DeclareMathOperator{\Cr}{C}
\DeclareMathOperator{\Hr}{H}
\DeclareMathOperator{\Ir}{I}
\DeclareMathOperator{\Cc}{\mathcal{C}}
\DeclareMathOperator{\car}{char}
\DeclareMathOperator{\cor}{cor}

\newcommand{\iso}{\xrightarrow{\sim}}
\newcommand{\F}{\mathbb{F}}
\newcommand{\Gb}{\mathbf{G}}
\newcommand{\Pb}{\mathbf{P}}
\newcommand{\Mb}{\mathbf{M}}
\newcommand{\Nb}{\mathbf{N}}
\newcommand{\Zb}{\mathbf{Z}}
\newcommand{\Bb}{\mathbf{B}}
\newcommand{\Sb}{\mathbf{S}}

\newcommand{\GL}{\mathrm{GL}}
\newcommand{\St}{\mathrm{St}}
\newcommand{\adm}{\mathrm{adm}}
\newcommand{\lfin}{\mathrm{l.fin}}
\newcommand{\crm}{\mathrm{c}}
\newcommand{\erm}{\mathrm{e}}

\hypersetup{pdfinfo={
	Title={On the exactness of ordinary parts over a local field of characteristic p},
	Author={Julien Hauseux},
	Subject={22E50},
	Keywords={local fields, reductive groups, admissible smooth representations, parabolic induction, ordinary parts, extensions}
}}

\title[On the exactness of ordinary parts]{On the exactness of ordinary parts over a local field of characteristic $p$}
\thanks{This research was partly supported by EPSRC grant EP/L025302/1.}

\author[J.~Hauseux]{Julien Hauseux}
\address{
Université de Lille\\
Département de Mathématiques\\
Cité scientifique, Bâtiment M2\\
59655 Villeneuve d'Ascq Cedex\\
France
}
\email{\href{mailto:julien.hauseux@math.univ-lille1.fr}{julien.hauseux@math.univ-lille1.fr}}

\begin{document}

\begin{abstract}
Let $G$ be a connected reductive group over a non-archimedean local field $F$ of residue characteristic $p$, $P$ be a parabolic subgroup of $G$, and $R$ be a commutative ring.
When $R$ is artinian, $p$ is nilpotent in $R$, and $\mathrm{char}(F)=p$, we prove that the ordinary part functor $\mathrm{Ord}_P$ is exact on the category of admissible smooth $R$-representations of $G$.
We derive some results on Yoneda extensions between admissible smooth $R$-representations of $G$.
\end{abstract}

\maketitle

\section{Results}

Let $F$ be a non-archimedean local field of residue characteristic $p$.
Let $\Gb$ be a connected reductive algebraic $F$-group and $G$ denote the topological group $\Gb(F)$.
We let $\Pb=\Mb\Nb$ be a parabolic subgroup of $\Gb$.
We write $\bar\Pb=\Mb\bar\Nb$ for the opposite parabolic subgroup.

Let $R$ be a commutative ring.
We write $\Mod_G^\infty(R)$ for the category of smooth $R$-representations of $G$ (i.e.\@ $R[G]$-modules $\pi$ such that for all $v \in \pi$ the stabiliser of $v$ is open in $G$) and $R[G]$-linear maps.
It is an $R$-linear abelian category.
When $R$ is noetherian, we write $\Mod_G^\adm(R)$ for the full subcategory of $\Mod_G^\infty(R)$ consisting of admissible representations (i.e.\@ those representations $\pi$ such that $\pi^H$ is finitely generated over $R$ for any open subgroup $H$ of $G$).
It is closed under passing to subrepresentations and extensions, thus it is an $R$-linear exact subcategory, but quotients of admissible representations may not be admissible when $\car(F)=p$ (see \cite[Example 4.4]{AHV}).

Recall the smooth parabolic induction functor $\Ind_{\bar P}^G : \Mod_M^\infty(R) \to \Mod_G^\infty(R)$, defined on any smooth $R$-representation $\sigma$ of $M$ as the $R$-module $\Ind_{\bar P}^G(\sigma)$ of locally constant functions $f : G \to \sigma$ satisfying $f(m\bar ng) = m \cdot f(g)$ for all $m \in M$, $\bar n \in \bar N$, and $g \in G$, endowed with the smooth action of $G$ by right translation.
It is $R$-linear, exact, and commutes with small direct sums.
In the other direction, there is the ordinary part functor $\Ord_P : \Mod_G^\infty(R) \to \Mod_M^\infty(R)$ (\cite{Em1,VigAdj}).
It is $R$-linear and left exact.
When $R$ is noetherian, $\Ord_P$ also commutes with small inductive limits, both functors respect admissibility, and the restriction of $\Ord_P$ to $\Mod_G^\adm(R)$ is right adjoint to the restriction of $\Ind_{\bar P}^G$ to $\Mod_M^\adm(R)$.

\begin{theo} \label{theo:exact}
If $R$ is artinian, $p$ is nilpotent in $R$, and $\car(F)=p$, then $\Ord_P$ is exact on $\Mod_G^\adm(R)$.
\end{theo}

Thus the situation is very different from the case $\car(F)=0$ (see \cite{Em2}).
On the other hand if $R$ is artinian and $p$ is invertible in $R$, then $\Ord_P$ is isomorphic on $\Mod_G^\adm(R)$ to the Jacquet functor with respect to $P$ (i.e.\@ the $N$-coinvariants) twisted by the inverse of the modulus character $\delta_P$ of $P$ (\cite[Corollary 4.19]{AHV}), so that it is exact on $\Mod_G^\adm(R)$ without any assumption on $\car(F)$.

\begin{rema}
Without any assumption on $R$, $\Ind_P^G : \Mod_M^\infty(R) \to \Mod_G^\infty(R)$ admits a left adjoint $\Lr_P^G : \Mod_G^\infty(R) \to \Mod_M^\infty(R)$ (the Jacquet functor with respect to $P$) and a right adjoint $\Rr_P^G : \Mod_G^\infty(R) \to \Mod_M^\infty(R)$ (\cite[Proposition 4.2]{VigAdj}).
If $R$ is noetherian and $p$ is nilpotent in $R$, then $\Rr_P^G$ is isomorphic to $\Ord_{\bar P}$ on $\Mod_G^\adm(R)$ (\cite[Corollary 4.13]{AHV}).
Thus under the assumptions of Theorem \ref{theo:exact}, $\Rr_P^G$ is exact on $\Mod_G^\adm(R)$.
On the other hand if $R$ is noetherian and $p$ is invertible in $R$, then $\Rr_P^G$ is expected to be isomorphic to $\delta_P \Lr_{\bar P}^G$ (`second adjointness'), and this is proved in the following cases: when $R$ is the field of complex numbers (\cite{Ber}) or an algebraically closed field of characteristic $\ell \neq p$ (\cite[II.3.8 2)]{VigLivre}); when $\Gb$ is a Levi subgroup of a general linear group or a classical group with $p \neq 2$ (\cite[Théorème 1.5]{Dat}); when $\Pb$ is a minimal parabolic subgroup of $\Gb$ (see also \cite{Dat}).
In particular, $\Lr_P^G$ and $\Rr_P^G$ are exact in all these cases.
\end{rema}

\begin{ques}
Are $\Lr_P^G$ and $\Rr_P^G$ exact when $R$ is noetherian, $p$ is nilpotent in $R$, and $\car(F)=p$?
\end{ques}

We derive from Theorem \ref{theo:exact} some results on Yoneda extensions between admissible $R$-representations of $G$.
We compute the $R$-modules $\Ext_G^\bullet$ in $\Mod_G^\adm(R)$.

\begin{coro} \label{coro:ext}
Assume $R$ artinian, $p$ nilpotent in $R$, and $\car(F)=p$.
Let $\sigma$ and $\pi$ be admissible $R$-representations of $M$ and $G$ respectively.
For all $n \geq 0$, there is a natural $R$-linear isomorphism
\begin{equation*}
\Ext_M^n(\sigma,\Ord_P(\pi)) \iso \Ext_G^n(\Ind_{\bar P}^G(\sigma),\pi).
\end{equation*}
\end{coro}

This is in contrast with the case $\car(F)=0$ (see \cite{JHD}).
A direct consequence of Corollary \ref{coro:ext} is that under the same assumptions, $\Ind_{\bar P}^G$ induces an isomorphism between the $\Ext^n$ for all $n \geq 0$ (Corollary \ref{coro:ind}).
When $R=C$ is an algebraically closed field of characteristic $p$ and $\car(F)=p$, we determine the extensions between certain irreducible admissible $C$-representations of $G$ using the classification of \cite{AHHV} (Proposition \ref{prop:triples}).
In particular, we prove that there exists no non-split extension of an irreducible admissible $C$-representation $\pi$ of $G$ by a supersingular $C$-representation of $G$ when $\pi$ is not the extension to $G$ of a supersingular representation of a Levi subgroup of $G$ (Corollary \ref{coro:sc}).
When $\Gb=\GL_2$, this was first proved by Hu (\cite[Theorem A.2]{Hu17}).

\subsection*{Acknowledgements}

The author would like to thank F.~Herzig and M.-F.~Vignéras for several comments on the first version of this paper.

\section{Proofs}

\subsection{Hecke action}

In this subsection, $\Mb$ denotes a linear algebraic $F$-group and $\Nb$ denotes a split unipotent algebraic $F$-group (see \cite[Appendix B]{CGP}) endowed with an action of $\Mb$ that we identify with the conjugation in $\Mb \ltimes \Nb$.
We fix an open submonoid $M^+$ of $M$ and a compact open subgroup $N_0$ of $N$ stable under conjugation by $M^+$.

If $\pi$ is a smooth $R$-representation of $M^+ \ltimes N_0$, then the $R$-modules $\Hr^\bullet(N_0,\pi)$, computed using the homogeneous cochain complex $\Cr^\bullet(N_0,\pi)$ (see \cite[§~I.2]{NSW}), are naturally endowed with the Hecke action of $M^+$, defined as the composite
\begin{equation*}
\Hr^\bullet(N_0,\pi) \xrightarrow{m} \Hr^\bullet(mN_0m^{-1},\pi) \xrightarrow{\cor} \Hr^\bullet(N_0,\pi)
\end{equation*}
for all $m \in M^+$.
At the level of cochains, this action is explicitly given as follows (see \cite[§~I.5]{NSW}).
We fix a set of representatives $\overline{N_0/mN_0m^{-1}} \subseteq N_0$ of the left cosets $N_0/mN_0m^{-1}$ and we write $n \mapsto \bar n$ for the projection $N_0 \twoheadrightarrow \overline{N_0/mN_0m^{-1}}$.
For $\phi \in \Cr^k(N_0,\pi)$, we have
\begin{multline} \label{lift}
(m \cdot \phi)(n_0,\dots,n_k) \\
= \sum_{\bar n \in \overline{N_0/mN_0m^{-1}}} \bar nm \cdot \phi(m^{-1}\bar n^{-1} n_0 \overline{n_0^{-1}\bar n}m,\dots,m^{-1}\bar n^{-1} n_k \overline{n_k^{-1}\bar n}m)
\end{multline}
for all $(n_0,\dots,n_k) \in N_0^{k+1}$.

\begin{lemm} \label{lemm:hecke}
Assume $p$ nilpotent in $R$ and $\car(F)=p$.
Let $\pi$ be a smooth $R$-representation of $M^+ \ltimes N_0$ and $m \in M^+$.
If the Hecke action $h_{N_0,m}$ of $m$ on $\pi^{N_0}$ is locally nilpotent (i.e.\@ for all $v \in \pi^{N_0}$ there exists $r \geq 0$ such that $h_{N_0,m}^r(v)=0$), then the Hecke action of $m$ on $\Hr^k(N_0,\pi)$ is locally nilpotent for all $k \geq 0$.
\end{lemm}

\begin{proof}
First, we prove the lemma when $pR=0$, i.e.\@ $R$ is a commutative $\F_p$-algebra.
We assume that the Hecke action of $m$ on $\pi^{N_0}$ is locally nilpotent and we prove the result together with the following fact: there exists a set of representatives $\overline{N_0/mN_0m^{-1}} \subseteq N_0$ of the left cosets $N_0/mN_0m^{-1}$ such that the action of
\begin{equation*}
S \coloneqq \sum_{\bar n \in \overline{N_0/mN_0m^{-1}}} \bar nm \in \F_p[M^+ \ltimes N_0]
\end{equation*}
on $\pi$ is locally nilpotent.

We proceed by induction on the dimension of $\Nb$ (recall that $\Nb$ is split so that it is smooth and connected).
If $\Nb=1$, then the (Hecke) action of $m$ on $\pi^{N_0}=\pi$ is locally nilpotent by assumption, so that the result and the fact are trivially true.
Assume $\Nb \neq 1$ and that the result and the fact are true for groups of smaller dimension.
Since $\Nb$ is split, it admits a non-trivial central subgroup isomorphic to the additive group.
We let $\Nb'$ be the subgroup of $\Nb$ generated by all such subgroups.
It is a non-trivial vector group (i.e.\@ isomorphic to a direct product of copies of the additive group) which is central (hence normal) in $\Nb$ and stable under conjugation by $\Mb$ (since it is a characteristic subgroup of $\Nb$).
We set $\Nb'' \coloneqq \Nb/\Nb'$.
It is a split unipotent algebraic $F$-group endowed with the induced action of $\Mb$ and $\dim(\Nb'') < \dim(\Nb)$.
Since $\Nb'$ is split, we have $N''=N/N'$.
We write $N'_0$ and $N''_0$ for the compact open subgroups $N' \cap N_0$ and $N_0/N'_0$ of $N'$ and $N''$ respectively.
They are stable under conjugation by $M^+$.
We fix a set-theoretic section $[-] : N''_0 \hookrightarrow N_0$.

Since $\Nb'$ is commutative and $p$-torsion, $N'_0$ is a compact $\F_p$-vector space.
Thus for any open subgroup $N'_1$ of $N'_0$, the short exact sequence of compact $\F_p$-vector spaces
\begin{equation*}
0 \to N'_1 \to N'_0 \to N'_0/N'_1 \to 0
\end{equation*}
splits.
Indeed, it admits an $\F_p$-linear splitting (since $\F_p$ is a field) which is automatically continuous (since $N'_0/N'_1$ is discrete).
In particular with $N'_1=mN'_0m^{-1}$, we may and do fix a section $N'_0/mN'_0m^{-1} \hookrightarrow N'_0$.
We write $\overline{N'_0/mN'_0m^{-1}}$ for its image, so that $N'_0 = \overline{N'_0/mN'_0m^{-1}} \times mN'_0m^{-1}$, and $n' \mapsto \bar n'$ for the projection $N'_0 \twoheadrightarrow \overline{N'_0/mN'_0m^{-1}}$.
We set
\begin{equation*}
S' \coloneqq \sum_{\bar n' \in \overline{N'_0/mN'_0m^{-1}}} \bar n'm \in \F_p[M^+ \ltimes N'_0].
\end{equation*}
For all $n'_0 \in N'_0$, we have $n'_0=\bar n'_0(\bar n_0'^{-1}n'_0)$ with $\bar n_0'^{-1}n'_0 \in mN'_0m^{-1}$, thus
\begin{equation*}
n'_0S' = \sum_{\bar n' \in \overline{N'_0/mN'_0m^{-1}}} (\bar n'_0\bar n')m(m^{-1}(\bar n_0'^{-1}n'_0)m) = S'(m^{-1}(\bar n_0'^{-1}n'_0)m)
\end{equation*}
with $m^{-1}(\bar n_0'^{-1}n'_0)m \in N'_0$ (in the first equality we use the fact that $N'_0$ is commutative and in the second one we use the fact that $\overline{N'_0/mN'_0m^{-1}}$ is a group).
Therefore, there is an inclusion $\F_p[N'_0] S' \subseteq S' \F_p[N'_0]$.

The $R$-module $\pi^{N'_0}$, endowed with the induced action of $N''_0$ and the Hecke action of $M^+$ with respect to $N'_0$, is a smooth $R$-representation of $M^+ \ltimes N''_0$ (see the proof of \cite[Lemme 3.2.1]{JH} in degree $0$).
On $\pi^{N'_0}$, the Hecke action of $m$ with respect to $N'_0$ coincides with the action of $S'$ by definition.
On $(\pi^{N'_0})^{N''_0} = \pi^{N_0}$, the Hecke action of $m$ with respect to $N''_0$ coincides with the Hecke action of $m$ with respect to $N_0$ (see the proof of \cite[Lemme 3.2.2]{JH}) which is locally nilpotent by assumption.
Thus by the induction hypothesis, there exists a set of representatives $\overline{N''_0/mN''_0m^{-1}} \subseteq N''_0$ of the left cosets $N''_0/mN''_0m^{-1}$ such that the action of
\begin{equation*}
S \coloneqq \sum_{\bar n'' \in \overline{N''_0/mN''_0m^{-1}}} [\bar n''] S' \in \F_p[M^+ \ltimes N_0]
\end{equation*}
on $\pi^{N'_0}$ is locally nilpotent.
Moreover, there is an inclusion $\F_p[N'_0] S \subseteq S \F_p[N'_0]$ (because $N'_0$ is central in $N_0$ and $\F_p[N'_0] S' \subseteq S' \F_p[N'_0]$).

We prove the fact.
By \cite[Lemme 2.1]{JHB},
\begin{equation*}
\overline{N_0/mN_0m^{-1}} \coloneqq \{[\bar n''] \bar n' : \bar n'' \in \overline{N''_0/mN''_0m^{-1}}, \bar n' \in \overline{N'_0/mN'_0m^{-1}}\} \subseteq N_0
\end{equation*}
is a set of representatives of the left cosets $N_0/mN_0m^{-1}$, and by definition
\begin{equation*}
S = \sum_{\bar n \in \overline{N_0/mN_0m^{-1}}} \bar nm.
\end{equation*}
We prove that the action of $S$ on $\pi$ is locally nilpotent.
We proceed as in the proof of \cite[Théorème 5.1 (i)]{Hu12}.
Let $v \in \pi$ and set $\pi_r \coloneqq \F_p[N'_0] \cdot (S^r \cdot v)$ for all $r \geq 0$.
Since $\F_p[N'_0] S \subseteq S \F_p[N'_0]$, we have $\pi_{r+1} \subseteq S \cdot \pi_r$ for all $r \geq 0$.
Since $N'_0$ is compact, we have $\dim_{\F_p}(\pi_r)<\infty$ for all $r \geq 0$.
If $S^r \cdot v \neq 0$, i.e.\@ $\pi_r \neq 0$, for some $r \geq 0$, then $\pi_r^{N'_0} \neq 0$ (because $N'_0$ is a pro-$p$ group and $\pi_r$ is a non-zero $\F_p$-vector space) so that $\dim_{\F_p}(S \cdot \pi_r)<\dim_{\F_p}\pi_r$ (because the action of $S$ on $\pi^{N'_0}$ is locally nilpotent).
Therefore $\pi_r=0$, i.e.\@ $S^r \cdot v = 0$, for all $r \geq \dim_{\F_p}(\pi_0)$.

We prove the result.
The $R$-modules $\Hr^\bullet(N'_0,\pi)$, endowed with the induced action of $N''_0$ and the Hecke action of $M^+$, are smooth $R$-representations of $M^+ \ltimes N''_0$ (see the proof of \cite[Lemme 3.2.1]{JH}\footnote{We do not know whether \cite[Proposition 2.1.11]{Em2} holds true when $\car(F)=p$, but \cite[Lemme 3.1.1]{JH} does and any injective object of $\Mod_{M^+ \ltimes N_0}^\infty(R)$ is still $N_0$-acyclic.\label{foot}}).
At the level of cochains, the actions of $n'' \in N''_0$ and $m$ are explicitly given as follows.
For $\phi \in \Cr^j(N'_0,\pi)$, we have
\begin{gather}
(n'' \cdot \phi)(n'_0,\dots,n'_j) = [n''] \cdot \phi(n'_0,\dots,n'_j) \label{lift''} \\
(m \cdot \phi)(n'_0,\dots,n'_j) = S' \cdot \phi(m^{-1} n'_0 \bar n_0'^{-1}m,\dots,m^{-1} n'_j \bar n_j'^{-1}m) \label{lift'}
\end{gather}
for all $(n'_0,\dots,n'_j) \in N_0'^{j+1}$ (for \eqref{lift''} we use the fact that $N'_0$ is central in $N_0$, for \eqref{lift'} we use \eqref{lift} and the fact that $n' \mapsto \bar n'$ is a group homomorphism $N'_0 \to \overline{N'_0/mN'_0m^{-1}}$).
Using \eqref{lift''} and \eqref{lift'}, we can give explicitly the Hecke action of $m$ on $\Hr^\bullet(N'_0,\pi)^{N''_0}$ at the level of cochains as follows.
For $\phi \in \Cr^j(N'_0,\pi)$, we have
\begin{equation*}
(m \cdot \phi)(n'_0,\dots,n'_j) = S \cdot \phi(m^{-1} n'_0 \bar n_0'^{-1}m,\dots,m^{-1} n'_j \bar n_j'^{-1}m)
\end{equation*}
for all $(n'_0,\dots,n'_j) \in N_0'^{j+1}$.
Since the action of $S$ on $\pi$ is locally nilpotent and the image of a locally constant cochain is finite by compactness of $N'_0$, we deduce that the Hecke action of $m$ on $\Hr^j(N'_0,\pi)^{N''_0}$ is locally nilpotent for all $j \geq 0$.
Thus the Hecke action of $m$ on $\Hr^i(N''_0,\Hr^j(N'_0,\pi))$ is locally nilpotent for all $i,j \geq 0$ by the induction hypothesis.
We conclude using the spectral sequence of smooth $R$-representations of $M^+$
\begin{equation*}
\Hr^i(N''_0,\Hr^j(N'_0,\pi)) \Rightarrow \Hr^{i+j}(N_0,\pi)
\end{equation*}
(see the proof of \cite[Proposition 3.2.3]{JH} and footnote \ref{foot}).

Now, we prove the lemma without assuming $pR=0$.
We proceed by induction on the degree of nilpotency $r$ of $p$ in $R$.
If $r \leq 1$, then the lemma is already proved.
We assume $r>1$ and that we know the lemma for rings in which the degree of nilpotency of $p$ is $r-1$.
There is a short exact sequence of smooth $R$-representations of $M^+ \ltimes N_0$
\begin{equation*}
0 \to p\pi \to \pi \to \pi/p\pi \to 0.
\end{equation*}
Taking the $N_0$-cohomology yields a long exact sequence of smooth $R$-representations of $M^+$
\begin{equation} \label{devcoh}
0 \to (p\pi)^{N_0} \to \pi^{N_0} \to (\pi/p\pi)^{N_0} \to \Hr^1(N_0,p\pi) \to \cdots.
\end{equation}
If the Hecke action of $m$ on $\pi^{N_0}$ is locally nilpotent, then the Hecke action of $m$ on $(p\pi)^{N_0}$ is also locally nilpotent so that the Hecke action of $m$ on $\Hr^k(N_0,p\pi)$ is locally nilpotent for all $k \geq 0$ by the induction hypothesis (since $p\pi$ is an $R/p^{r-1}R$-module).
Using \eqref{devcoh}, we deduce that the Hecke action of $m$ on $(\pi/p\pi)^{N_0}$ is also locally nilpotent so that the Hecke action of $m$ on $\Hr^k(N_0,\pi/p\pi)$ is locally nilpotent for all $k \geq 0$ (since $\pi/p\pi$ is an $\F_p$-vector space).
Using again \eqref{devcoh}, we conclude that the Hecke action of $m$ on $\Hr^k(N_0,\pi)$ is locally nilpotent for all $k \geq 0$.
\end{proof}

\subsection{Proof of the main result}

We fix a compact open subgroup $N_0$ of $N$ and we let $M^+$ be the open submonoid of $M$ consisting of those elements $m$ contracting $N_0$ (i.e.\@ $mN_0m^{-1} \subseteq N_0$).
We let $\Zb_\Mb$ denote the centre of $\Mb$ and we set $Z_M^+ \coloneqq Z_M \cap M^+$.
We fix an element $z \in Z_M^+$ strictly contracting $N_0$ (i.e.\@ $\cap_{r \geq0} z^rN_0z^{-r}=1$).

Recall that the ordinary part of a smooth $R$-representation $\pi$ of $P$ is the smooth $R$-representation of $M$
\begin{equation*}
\Ord_P(\pi) \coloneqq (\Ind_{M^+}^M(\pi^{N_0}))^{Z_M-\lfin}
\end{equation*}
where $\Ind_{M^+}^M(\pi^{N_0})$ is defined as the $R$-module of functions $f : M \to \pi^{N_0}$ such that $f(mm') = m \cdot f(m')$ for all $m \in M^+$ and $m' \in M$, endowed with the action of $M$ by right translation, and the superscript $^{Z_M-\lfin}$ denotes the subrepresentation consisting of locally $Z_M$-finite elements (i.e.\@ those elements $f$ such that $R[Z_M] \cdot f$ is contained in a finitely generated $R$-submodule).
The action of $M$ on the latter is smooth by \cite[Remark 7.6]{VigAdj}.
If $R$ is artinian and $\pi^{N_0}$ is locally $Z_M^+$-finite (i.e.\@ it may be written as the union of finitely generated $Z_M^+$-invariant $R$-submodules), then there is a natural $R$-linear isomorphism
\begin{equation} \label{loc}
\Ord_P(\pi) \iso R[z^{\pm1}] \otimes_{R[z]} \pi^{N_0}
\end{equation}
(cf.\@ \cite[Lemma 3.2.1 (1)]{Em2}, whose proof also works when $\car(F)=p$ and over any artinian ring).

If $\sigma$ is a smooth $R$-representation of $M$, then the $R$-module $\Cc_\crm^\infty(N,\sigma)$ of locally constant functions $f : N \to \sigma$ with compact support, endowed with the action of $N$ by right translation and the action of $M$ given by $(m \cdot f) : n \mapsto m \cdot f(m^{-1}nm)$ for all $m \in M$, is a smooth $R$-representation of $P$.
Thus we obtain a functor $\Cc_\crm^\infty(N,-) : \Mod_M^\infty(R) \to \Mod_P^\infty(R)$.
It is $R$-linear, exact, and commutes with small direct sums.
The results of \cite[§~4.2]{Em1} hold true when $\car(F)=p$ and over any ring, thus the functors
\begin{gather*}
\Cc_\crm^\infty(N,-) : \Mod_M^\infty(R)^{Z_M-\lfin} \to \Mod_P^\infty(R) \\
\Ord_P : \Mod_P^\infty(R) \to \Mod_M^\infty(R)^{Z_M-\lfin}
\end{gather*}
are adjoint and the unit of the adjunction is an isomorphism.

\begin{lemm} \label{lemm:lnil}
Assume $R$ artinian, $p$ nilpotent in $R$, and $\car(F)=p$.
Let $\pi$ be a smooth $R$-representation of $P$.
If $\pi^{N_0}$ is locally $Z_M^+$-finite, then the Hecke action of $z$ on $\Hr^k(N_0,\pi)$ is locally nilpotent for all $k \geq 1$.
\end{lemm}

\begin{proof}
We set $\sigma \coloneqq \Ord_P(\pi)$.
The counit of the adjunction between $\Cc_\crm^\infty(N,-)$ and $\Ord_P$ induces a natural morphism of smooth $R$-representations of $P$
\begin{equation} \label{nat}
\Cc_\crm^\infty(N,\sigma) \to \pi.
\end{equation}
Taking the $N_0$-invariants yields a morphism of smooth $R$-representations of $M^+$
\begin{equation} \label{natinv}
\Cc_\crm^\infty(N,\sigma)^{N_0} \to \pi^{N_0}.
\end{equation}
By definition, $\sigma$ is locally $Z_M$-finite so it may be written as the union of finitely generated $Z_M$-invariant $R$-submodules $(\sigma_i)_{i \in I}$.
Thus $\Cc_\crm^\infty(N,\sigma)^{N_0}$ is the union of the finitely generated $Z_M^+$-invariant $R$-submodules $(\Cc^\infty(z^{-r}N_0z^r,\sigma_i)^{N_0})_{r \geq 0,i \in I}$, so it is locally $Z_M^+$-finite.
By assumption, $\pi^{N_0}$ is also locally $Z_M^+$-finite.
Therefore, using \eqref{loc} and its analogue with $\Cc_\crm^\infty(N,\sigma)$ instead of $\pi$, the localisation with respect to $z$ of \eqref{natinv} is the natural morphism of smooth $R$-representations of $M$
\begin{equation*}
\Ord_P(\Cc_\crm^\infty(N,\sigma)) \to \Ord_P(\pi)
\end{equation*}
induced by applying the functor $\Ord_P$ to \eqref{nat}, and it is an isomorphism since the unit of the adjunction between $\Cc_\crm^\infty(N,-)$ and $\Ord_P$ is an isomorphism.

Let $\kappa$ (resp.\@ $\iota$) be the kernel (resp.\@ image) of \eqref{nat}, hence two short exact sequences of smooth $R$-representations of $P$
\begin{gather}
0 \to \kappa \to \Cc_\crm^\infty(N,\sigma) \to \iota \to 0 \label{ses1} \\
0 \to \iota \to \pi \to \pi/\iota \to 0 \label{ses2}
\end{gather}
such that the third arrow of \eqref{ses1} and the second arrow of \eqref{ses2} fit into a commutative diagram of smooth $R$-representations of $P$
\begin{equation*}
\begin{tikzcd}
\Cc_\crm^\infty(N,\sigma) \drar[two heads] \ar[rr] && \pi \\
&\iota \urar[hook]&
\end{tikzcd}
\end{equation*}
whose upper arrow is \eqref{nat}.
Taking the $N_0$-invariants yields a commutative diagram of smooth $R$-representations of $M^+$
\begin{equation*}
\begin{tikzcd}
\Cc_\crm^\infty(N,\sigma)^{N_0} \drar \ar[rr] && \pi^{N_0} \\
&\iota^{N_0} \urar[hook]&
\end{tikzcd}
\end{equation*}
whose upper arrow is \eqref{natinv}.
Since the localisation with respect to $z$ of the latter is an isomorphism, the localisation with respect to $z$ of the injection $\iota^{N_0} \hookrightarrow \pi^{N_0}$ is surjective, thus it is an isomorphism (as it is also injective by exactness of localisation).
Therefore the localisation with respect to $z$ of the morphism $\Cc_\crm^\infty(N,\sigma)^{N_0} \to \iota^{N_0}$ is an isomorphism.

Since $\Cc_\crm^\infty(N,\sigma) \cong \bigoplus_{n \in N/N_0} \Cc^\infty(nN_0,\sigma)$ as a smooth $R$-representation of $N_0$, it is $N_0$-acyclic (see \cite[§~I.3]{NSW}).
Thus the long exact sequence of $N_0$-cohomology induced by \eqref{ses1} yields an exact sequence of smooth $R$-representations of $M^+$
\begin{equation} \label{les1a}
0 \to \kappa^{N_0} \to \Cc_\crm^\infty(N,\sigma)^{N_0} \to \iota^{N_0} \to \Hr^1(N_0,\kappa) \to 0
\end{equation}
and an isomorphism of smooth $R$-representations of $M^+$
\begin{equation} \label{les1b}
\Hr^k(N_0,\iota) \iso \Hr^{k+1}(N_0,\kappa)
\end{equation}
for all $k \geq 1$.
Since the localisation with respect to $z$ of the third arrow of \eqref{les1a} is an isomorphism, the Hecke action of $z$ on $\kappa^{N_0}$ is locally nilpotent.
Thus the Hecke action of $z$ on $\Hr^k(N_0,\kappa)$ is locally nilpotent for all $k \geq 0$ by Lemma \ref{lemm:hecke}.
Using \eqref{les1b}, we deduce that the Hecke action of $z$ on $\Hr^k(N_0,\iota)$ is locally nilpotent for all $k \geq 1$.

Taking the $N_0$-cohomology of \eqref{ses2} yields a long exact sequence of smooth $R$-representations of $M^+$
\begin{equation} \label{les2}
0 \to \iota^{N_0} \to \pi^{N_0} \to (\pi/\iota)^{N_0} \to \Hr^1(N_0,\iota) \to \cdots.
\end{equation}
Since the localisation with respect to $z$ of the second arrow is an isomorphism and the Hecke action of $z$ on $\Hr^1(N_0,\iota)$ is locally nilpotent, the Hecke action of $z$ on $(\pi/\iota)^{N_0}$ is locally nilpotent.
Thus the Hecke action of $z$ on $\Hr^k(N_0,\pi/\iota)$ is locally nilpotent for all $k \geq 0$ by Lemma \ref{lemm:hecke}.
We conclude using \eqref{les2} and the fact that the Hecke action of $z$ on $\Hr^k(N_0,\iota)$ is locally nilpotent for all $k \geq 1$.
\end{proof}

\begin{proof}[Proof of Theorem \ref{theo:exact}]
Assume $R$ artinian, $p$ nilpotent in $R$, and $\car(F)=p$.
Let
\begin{equation} \label{sespi}
0 \to \pi_1 \to \pi_2 \to \pi_3 \to 0
\end{equation}
be a short exact sequence of admissible $R$-representations of $G$.
Taking the $N_0$-invariants yields an exact sequence of smooth $R$-representations of $M^+$
\begin{equation} \label{sespicoh}
0 \to \pi_1^{N_0} \to \pi_2^{N_0} \to \pi_3^{N_0} \to \Hr^1(N_0,\pi_1).
\end{equation}
The terms $\pi_1^{N_0}, \pi_2^{N_0}, \pi_3^{N_0}$ are locally $Z_M^+$-finite (cf.\@ \cite[Theorem 3.4.7 (1)]{Em2}, whose proof in degree $0$ also works when $\car(F)=p$ and over any noetherian ring) and the Hecke action of $z$ on $\Hr^1(N_0,\pi_1)$ is locally nilpotent by Lemma \ref{lemm:lnil}.
Therefore, using \eqref{loc}, the localisation with respect to $z$ of \eqref{sespicoh} is the short sequence of admissible $R$-representations of $M$
\begin{equation*}
0 \to \Ord_P(\pi_1) \to \Ord_P(\pi_2) \to \Ord_P(\pi_3) \to 0
\end{equation*}
induced by applying the functor $\Ord_P$ to \eqref{sespi}, and it is exact by exactness of localisation.
\end{proof}

\subsection{Results on extensions}

We assume $R$ noetherian.
The $R$-linear category $\Mod_G^\adm(R)$ is not abelian in general, but merely exact in the sense of Quillen (\cite{Qui}).
An exact sequence of admissible $R$-representations of $G$ is an exact sequence of smooth $R$-representations of $G$
\begin{equation*}
\cdots \to \pi_{n-1} \to \pi_n \to \pi_{n+1} \to \cdots
\end{equation*}
such that the kernel and the cokernel of every arrow are admissible.
In particular, each term of the sequence is also admissible.

For $n \geq 0$ and $\pi,\pi'$ two admissible $R$-representations of $G$, we let $\Ext_G^n(\pi',\pi)$ denote the $R$-module of $n$-fold Yoneda extensions (\cite{Yon}) of $\pi'$ by $\pi$ in $\Mod_G^\adm(R)$, defined as equivalence classes of exact sequences
\begin{equation*}
0 \to \pi \to \pi_1 \to \cdots \to \pi_n \to \pi' \to 0.
\end{equation*}

We let $D(G)$ denote the derived category of $\Mod_G^\adm(R)$ (\cite{Nee,Kel,Buh}).
The results of \cite[§~III.3.2]{Ver} on the Yoneda construction carry over to this setting (see e.g.\@ \cite[Proposition A.13]{Pos}), hence a natural $R$-linear isomorphism
\begin{equation*}
\Ext_G^n(\pi',\pi) \cong \Hom_{D(G)}(\pi',\pi[n]).
\end{equation*}

\begin{proof}[Proof of Corollary \ref{coro:ext}]
Since $\Ind_{\bar P}^G$ and $\Ord_P$ are exact adjoint functors between $\Mod_M^\adm(R)$ and $\Mod_G^\adm(R)$ by Theorem \ref{theo:exact}, they induce adjoint functors between $D(M)$ and $D(G)$, hence natural $R$-linear isomorphisms
\begin{align*}
\Ext_M^n(\sigma,\Ord_P(\pi)) &\cong \Hom_{D(M)}(\sigma,\Ord_P(\pi)[n]) \\
&\cong \Hom_{D(G)}(\Ind_{\bar P}^G(\sigma),\pi[n]) \\
&\cong \Ext_G^n(\Ind_{\bar P}^G(\sigma),\pi)
\end{align*}
for all $n \geq 0$.
\end{proof}

\begin{rema}
We give a more explicit proof of Corollary \ref{coro:ext}.
The exact functor $\Ind_{\bar P}^G$ and the counit of the adjunction between $\Ind_{\bar P}^G$ and $\Ord_P$ induce an $R$-linear morphism
\begin{equation} \label{indext}
\Ext_M^n(\sigma,\Ord_P(\pi)) \to \Ext_G^n(\Ind_{\bar P}^G(\sigma),\pi).
\end{equation}
In the other direction, the exact (by Theorem \ref{theo:exact}) functor $\Ord_P$ and the unit of the adjunction between $\Ind_{\bar P}^G$ and $\Ord_P$ induce an $R$-linear morphism
\begin{equation} \label{ordext}
\Ext_G^n(\Ind_{\bar P}^G(\sigma),\pi) \to \Ext_M^n(\sigma,\Ord_P(\pi)).
\end{equation}
We prove that \eqref{ordext} is the inverse of \eqref{indext}.
For $n=0$ this is the unit-counit equations.
Assume $n \geq 1$ and let
\begin{equation} \label{extsigma}
0 \to \Ord_P(\pi) \to \sigma_1 \to \cdots \to \sigma_n \to \sigma \to 0
\end{equation}
be an exact sequence of admissible $R$-representations of $M$.
By \cite[§~3]{Yon}, the image of the class of \eqref{extsigma} under \eqref{indext} is the class of any exact sequence of admissible $R$-representations of $G$
\begin{equation} \label{extpi}
0 \to \pi \to \pi_1 \to \cdots \to \pi_n \to \Ind_{\bar P}^G(\sigma) \to 0
\end{equation}
such that there exists a commutative diagram of admissible $R$-representations of $G$
\begin{equation*}
\begin{tikzcd}[column sep=small]
0 \rar & \Ind_{\bar P}^G(\Ord_P(\pi)) \dar \rar & \Ind_{\bar P}^G(\sigma_1) \dar \rar & \cdots \rar & \Ind_{\bar P}^G(\sigma_n) \dar \rar & \Ind_{\bar P}^G(\sigma) \dar[equal] \rar & 0 \\
0 \rar & \pi \rar & \pi_1 \rar & \cdots \rar & \pi_n \rar & \Ind_{\bar P}^G(\sigma) \rar & 0
\end{tikzcd}
\end{equation*}
in which the upper row is obtained from \eqref{extsigma} by applying the exact functor $\Ind_{\bar P}^G$, the lower row is \eqref{extpi}, and the leftmost vertical arrow is the natural morphism induced by the counit of the adjunction between $\Ind_{\bar P}^G$ and $\Ord_P$.
Applying the exact functor $\Ord_P$ to the diagram and using the unit of the adjunction between $\Ind_{\bar P}^G$ and $\Ord_P$ yields a commutative diagram of admissible $R$-representations of $M$
\begin{equation*}
\begin{tikzcd}[column sep=small]
0 \rar &\Ord_P(\pi) \dar[equal] \rar & \sigma_1 \dar \rar & \cdots \rar & \sigma_n \dar \rar & \sigma \dar \rar & 0 \\
0 \rar & \Ord_P(\pi) \rar & \Ord_P(\pi_1) \rar & \cdots \rar & \Ord_P(\pi_n) \rar & \Ord_P(\Ind_{\bar P}^G(\sigma)) \rar & 0
\end{tikzcd}
\end{equation*}
in which the lower row is obtained from \eqref{extpi} by applying the exact functor $\Ord_P$, the upper row is \eqref{extsigma}, and the rightmost vertical arrow is the natural morphism induced by the unit of the adjunction between $\Ind_{\bar P}^G$ and $\Ord_P$.
The leftmost vertical morphism is the identity by the unit-counit equations.
Thus the image of the class of \eqref{extpi} under \eqref{ordext} is the class of \eqref{extsigma} by \cite[§~3]{Yon}.
We have proved that \eqref{ordext} is a left inverse of \eqref{indext}.
The proof that it is a right inverse is dual.
\end{rema}

\begin{coro} \label{coro:ind}
Assume $R$ artinian, $p$ nilpotent in $R$, and $\car(F)=p$.
Let $\sigma$ and $\sigma'$ be two admissible $R$-representations of $M$.
The functor $\Ind_{\bar P}^G$ induces an $R$-linear isomorphism
\begin{equation*}
\Ext_M^n(\sigma',\sigma) \iso \Ext_G^n(\Ind_{\bar P}^G(\sigma'),\Ind_{\bar P}^G(\sigma))
\end{equation*}
for all $n \geq 0$.
\end{coro}

\begin{proof}
The isomorphism in the statement is the composite
\begin{equation*}
\Ext_M^n(\sigma',\sigma) \iso \Ext_M^n(\sigma',\Ord_P(\Ind_{\bar P}^G(\sigma))) \iso \Ext_G^n(\Ind_{\bar P}^G(\sigma'),\Ind_{\bar P}^G(\sigma))
\end{equation*}
where the first isomorphism is induced by the unit of the adjunction between $\Ind_{\bar P}^G$ and $\Ord_P$, which is an isomorphism, and the second one is the isomorphism of Corollary \ref{coro:ext} with $\sigma'$ and $\Ind_{\bar P}^G(\sigma)$ instead of $\sigma$ and $\pi$ respectively.
\end{proof}

We fix a minimal parabolic subgroup $\Bb \subseteq \Gb$, a maximal split torus $\Sb \subseteq \Bb$, and we write $\Delta$ for the set of simple roots of $\Sb$ in $\Bb$.
We say that a parabolic subgroup $\Pb=\Mb\Nb$ of $\Gb$ is \emph{standard} if $\Bb \subseteq \Pb$ and $\Sb \subseteq \Mb$.
In this case, we write $\Delta_\Pb$ for the corresponding subset of $\Delta$, and given $\alpha \in \Delta_\Pb$ (resp.\@ $\alpha \in \Delta \backslash \Delta_\Pb$) we write $\Pb^\alpha=\Mb^\alpha\Nb^\alpha$ (resp.\@ $\Pb_\alpha=\Mb_\alpha\Nb_\alpha$) for the standard parabolic subgroup corresponding to $\Delta_\Pb \backslash \{\alpha\}$ (resp.\@ $\Delta_\Pb \sqcup \{\alpha\}$).

Let $C$ be an algebraically closed field of characteristic $p$.
Given a standard parabolic subgroup $P=MN$ and a smooth $C$-representation $\sigma$ of $M$, there exists a largest standard parabolic subgroup $P(\sigma)=M(\sigma)N(\sigma)$ such that the inflation of $\sigma$ to $P$ extends to a smooth $C$-representation $\prescript{\erm}{}{\sigma}$ of $P(\sigma)$, and this extension is unique (\cite[II.7 Corollary 1]{AHHV}).
We say that a smooth $C$-representation of $G$ is \emph{supercuspidal} if it is irreducible, admissible, and does not appear as a subquotient of $\Ind_P^G(\sigma)$ for any proper parabolic subgroup $P=MN$ of $G$ and any irreducible admissible $C$-representation $\sigma$ of $M$.
A \emph{supercuspidal standard $C[G]$-triple} is a triple $(P,\sigma,Q)$ where $P=MN$ is a standard parabolic subgroup, $\sigma$ is a supercuspidal $C$-representation of $M$, and $Q$ is a parabolic subgroup of $G$ such that $P \subseteq Q \subseteq P(\sigma)$.
To such a triple is attached in \cite{AHHV} a smooth $C$-representation of $G$
\begin{equation*}
\Ir_G(P,\sigma,Q) \coloneqq \Ind_{P(\sigma)}^G(\prescript{\erm}{}{\sigma} \otimes \St_Q^{P(\sigma)})
\end{equation*}
where $\St_Q^{P(\sigma)} \coloneqq \Ind_Q^{P(\sigma)}(1) / \sum_{Q \subsetneq Q' \subseteq P(\sigma)} \Ind_{Q'}^{P(\sigma)}(1)$ (here $1$ denotes the trivial $C$-representation) is the inflation to $P(\sigma)$ of the generalised Steinberg representation of $M(\sigma)$ with respect to $M(\sigma) \cap Q$ (\cite{GK,LySt}).
It is irreducible and admissible (\cite[I.3 Theorem 1]{AHHV}).

\begin{prop} \label{prop:triples}
Assume $\car(F)=p$.
Let $(P,\sigma,Q)$ and $(P',\sigma',Q')$ be two supercuspidal standard $C[G]$-triples.
If $Q \not \subseteq Q'$, then the $C$-vector space
\begin{equation*}
\Ext_G^1(\Ir_G(P',\sigma',Q'),\Ir_G(P,\sigma,Q))
\end{equation*}
is non-zero if and only if $P'=P$, $\sigma' \cong \sigma$, and $Q'=Q^\alpha$ for some $\alpha \in \Delta_Q$, in which case it is one-dimensional and the unique (up to isomorphism) non-split extension of $\Ir_G(P',\sigma',Q')$ by $\Ir_G(P,\sigma,Q)$ is the admissible $C$-representation of $G$
\begin{equation*}
\Ind_{P(\sigma)^\alpha}^G(\Ir_{M(\sigma)^\alpha}(M(\sigma)^\alpha \cap P,\sigma,M(\sigma)^\alpha \cap Q)).
\end{equation*}
\end{prop}

\begin{proof}
There is a natural short exact sequence of admissible $C$-representations of $G$
\begin{equation} \label{I}
\textstyle
0 \to \sum_{Q' \subsetneq Q'' \subseteq P(\sigma')} \Ind_{Q''}^G(\sigma') \to \Ind_{Q'}^G(\sigma') \to \Ir_G(P',\sigma',Q') \to 0.
\end{equation}
Note that we can restrict the sum to those $Q''$ that are minimal, i.e.\@ of the form $Q'_\alpha$ for some $\alpha \in \Delta_{P(\sigma')} \backslash \Delta_{Q'}$.
Moreover, we deduce from \cite[Theorem 3.2]{AHV} that its cosocle is isomorphic to $\bigoplus_{\alpha \in \Delta_{P(\sigma')} \backslash \Delta_{Q'}} \Ir_G(P',\sigma',Q'_\alpha)$.
Now if $Q \not \subseteq Q'$, then $\Ord_{\bar Q'}(\Ir_G(P,\sigma,Q))=0$ by \cite[Theorem 1.1 (ii) and Corollary 4.13]{AHV} so that using Corollary \ref{coro:ext}, we see that the long exact sequence of Yoneda extensions obtained by applying the functor $\Hom_G(-,\Ir_G(P,\sigma,Q))$ to \eqref{I} yields a natural $C$-linear isomorphism
\begin{multline*}
\textstyle
\Ext_G^{n-1}(\sum_{Q' \subsetneq Q'' \subseteq P(\sigma')} \Ind_{Q''}^G(\sigma'),\Ir_G(P,\sigma,Q)) \\
\iso \Ext_G^n(\Ir_G(P',\sigma',Q'),\Ir_G(P,\sigma,Q))
\end{multline*}
for all $n \geq 1$.
In particular, with $n=1$ and using the identification of the cosocle of the sum and \cite[I.3 Theorem 2]{AHHV}, we deduce that the $C$-vector space in the statement is non-zero if and only if $P'=P$, $\sigma' \cong \sigma$, and $Q=Q'_\alpha$ for some $\alpha \in \Delta_{P(\sigma')} \backslash \Delta_{Q'}$ (or equivalently $Q'=Q^\alpha$ for some $\alpha \in \Delta_Q$), in which case it is one-dimensional.
Finally, using again \cite[Theorem 3.2]{AHV}, we see that for all $\alpha \in \Delta_Q$ the admissible $C$-representation of $G$ in the statement is a non-split extension of $\Ir_G(P,\sigma,Q^\alpha)$ by $\Ir_G(P,\sigma,Q)$.
\end{proof}

\begin{coro} \label{coro:sc}
Assume $\car(F)=p$.
Let $\pi$ and $\pi'$ be two irreducible admissible $C$-representations of $G$.
If $\pi$ is supercuspidal and $\pi'$ is not the extension to $G$ of a supercuspidal representation of a Levi subgroup of $G$, then $\Ext_G^1(\pi',\pi)=0$.
\end{coro}

\begin{proof}
By \cite[I.3 Theorem 3]{AHHV}, there exist two supercuspidal standard $C[G]$-triples $(P,\sigma,Q)$ and $(P',\sigma',Q')$ such that $\pi \cong \Ir_G(P,\sigma,Q)$ and $\pi' \cong \Ir_G(P',\sigma',Q')$.
The assumptions on $\pi$ and $\pi'$ are equivalent to $P=G$ and $Q' \neq G$.
In particular, $Q \not \subseteq Q'$ and $P \neq P'$ so that $\Ext_G^1(\pi',\pi)=0$ by Proposition \ref{prop:triples}.
\end{proof}

\bibliographystyle{amsalpha}
\bibliography{exactitude}

\end{document}